\documentclass[11pt]{amsart}
\usepackage{amssymb, amsfonts, amsmath}
\usepackage{latexsym}
\usepackage{graphicx}
\usepackage{subcaption}
\usepackage{tikz}
\usepackage{url}

\theoremstyle{plain}
\newtheorem{theorem}{Theorem}
\newtheorem{corollary}{Corollary}
\newtheorem{lemma}{Lemma}
\newtheorem{proposition}{Proposition}

\newtheorem{propcase}{Case}
\newtheorem{theoremcase}{Case}
\newtheorem{lemmacase}{Case}

\theoremstyle{definition}
\newtheorem{definition}{Definition}

\newtheorem{example}{Example}

\numberwithin{equation}{section}
\numberwithin{theorem}{section}
\numberwithin{definition}{section}
\numberwithin{lemma}{section}
\numberwithin{corollary}{section}
\numberwithin{proposition}{section}
\numberwithin{remark}{section}
\numberwithin{example}{section}
\numberwithin{table}{section}
\numberwithin{theoremcase}{theorem}
\numberwithin{propcase}{proposition}
\numberwithin{lemmacase}{lemma}

\allowdisplaybreaks

\begin{document}
	
	\title[Magic Polygons and Their Properties]{Magic Polygons and Their Properties}
	
	\author[V. Jakicic ]{Victoria Jakicic}
	\author[R. Bouchat]{Rachelle Bouchat}
	
	\address{Department of Mathematics, Indiana University of Pennsylvania, Indiana, PA 15701, USA}
	\email{gkwt@iup.edu}
	\address{Department of Mathematics, Indiana University of Pennsylvania, Indiana, PA 15701, USA}
	\email{rbouchat@iup.edu}
	
	\begin{abstract} 
		Magic squares are arrangements of natural numbers into square arrays, where the sum of each row, each column, and both diagonals is the same.  In this paper, the concept of a magic square with 3 rows and 3 columns is generalized to define magic polygons.  Furthermore, this paper will examine the existence of magic polygons, along with several other properties inherent to magic polygons.
	\end{abstract}

	\maketitle\thispagestyle{empty}
	
	%\noindent{\bf Keywords:} Fractional derivatives, concavity.\\
	%{\bf AMS Subject Classification:}
	
	%-------------------------
	\section{Introduction}
	%-------------------------
	
	In recreational mathematics, magic squares have intrigued both mathematicians and the public alike. A traditional \emph{magic square} is an $n \times n$ array filled with the numbers $1$ to $n^2$ with the requirement that each of the $n$ rows, $n$ columns, and two diagonals sum to the same total, or \emph{magic sum}. Chinese mathematicians knew of magic squares in 650 BC, and the first appearance of the Lo Shu square was in 1200 CE (see~\cite{Andrews}).
	
	The existence of magic squares has been studied, and in 1942 (see~\cite{Kraitchik}), Kraitchik produced techniques for forming magic squares for $n$ even or odd.  Additionally, the number of solutions for various sizes of magic squares has been studied.  When $n = 1$, the magic square is trivial with its only entry being $1$, and there does not exist a $2 \times 2$ magic square. When $n = 3$, there is a unique (up to symmetry) magic square (see~\cite{Swetz}), and the  number of solutions for $4 \times 4$ and $5 \times 5$ magic squares are $880$ and $275,305,224$, respectively (see~\cite{Loly}). 
	
	Many variations of magic squares have also been studied by placing restrictions on the set of numbers used to fill in the magic square or placing additional requirements on the magic sum.  For a few of these variations, see~\cite{Kraitchik2},~\cite{Lucas}, and~\cite{Tarry}.  This paper aims to extend the idea of a traditional $3\times 3$ magic square to any regular polygon.

	\begin{definition}
		Consider an $n$-sided regular polygon with nodes placed at each of the vertices, the midpoints of each side, and the center of the polygon.  Then, a \textit{magic polygon} is an $n$-sided regular polygon containing an arrangement of the numbers $1$ to $2n+1$ within this arrangement of nodes such that the sum of the three nodes on each side and the sum of the three nodes on each diagonal are equal.  This sum is called the \textit{magic sum}. 
	\end{definition}
	
	This definition is a natural generalization of the concept of a magic square as can be illustrated with the $3\times 3$ magic square below:
	
	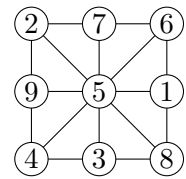
\begin{figure}[h]
		\begin{subfigure}[!h]{.4\textwidth}
			\centering
			\begin{tikzpicture}[scale=.45]
			\draw (0,0)--(4,0)--(4,4)--(0,4)--(0,0);
			\draw (1.33,0)--(1.33,4);
			\draw (2.66,0)--(2.66,4);
			\draw (0,1.33)--(4,1.33);
			\draw (0,2.66)--(4,2.66);
			
			\node at (1.995,1.995) {$5$};
			\node at (0.665,0.665) {$4$};
			\node at (1.995,0.665) {$3$};
			\node at (3.333,0.665) {$8$};
			\node at (0.665,1.995) {$9$};
			\node at (0.665,3.333) {$2$};
			\node at (3.333,1.995) {$1$};
			\node at (1.995,3.333) {$7$};
			\node at (3.333,3.333) {$6$};
			\end{tikzpicture}
			\subcaption{A traditional magic square}
			\label{fig:magicsquare}
		\end{subfigure}
		\begin{subfigure}[!h]{.5\textwidth}
			\centering
			\begin{tikzpicture}[scale=.45]
			\draw (0,0)--(4,0)--(4,4)--(0,4)--(0,0);
			\draw (0,2)--(4,2);
			\draw (2,0)--(2,4);
			\draw (0,4)--(4,0);
			\draw (0,0)--(4,4);
			\draw[fill=white] (0,4) circle (14pt);
			\draw[fill=white] (0,2) circle (14pt);
			\draw[fill=white] (0,0) circle (14pt);
			\draw[fill=white] (2,0) circle (14pt);
			\draw[fill=white] (2,2) circle (14pt);
			\draw[fill=white] (2,4) circle (14pt);
			\draw[fill=white] (4,0) circle (14pt);
			\draw[fill=white] (4,2) circle (14pt);
			\draw[fill=white] (4,4) circle (14pt);
			
			\node at (2,2) {$5$};
			\node at (0,0) {$4$};
			\node at (0,2) {$9$};
			\node at (0,4) {$2$};
			\node at (2,0) {$3$};
			\node at (2,4) {$7$};
			\node at (4,0) {$8$};
			\node at (4,2) {$1$};
			\node at (4,4) {$6$};
			\end{tikzpicture}
			\subcaption{A magic square viewed as a magic polygon}
			\label{fig:magicpolygonsquare}
		\end{subfigure}
		\caption{Visualizing the transition from magic square to magic polygon}
	\end{figure}

	\begin{example}\label{hexagon}
		For the magic hexagon depicted below, the magic sum is $m = 21$.  The sum of the three numbers on each of the sides is 21, for example $10+9+2=21$; and the sum of the three numbers on each of the diagonals is also 21, for example $5+7+9=21$.
		
		\begin{figure}[h]
			\centering
			\begin{tikzpicture}[scale=.65]
			\draw (0,0)--(2,0)--(3.17,2)--(2,4)--(0,4)--(-1.17,2)--(0,0);
			\draw (1,0)--(1,4);
			\draw (2,0)--(0,4);
			\draw (2.67,1)--(-.67,3.17);
			\draw (3.17,2)--(-1.17,2);
			\draw (2.67,3)--(-.67,1);
			\draw (0,0)--(2,4);
			\draw[fill=white] (0,0) circle (10pt);
			\draw[fill=white] (2,0) circle (10pt);
			\draw[fill=white] (3.17,2) circle (10pt);
			\draw[fill=white] (2,4) circle (10pt);
			\draw[fill=white] (0,4) circle (10pt);
			\draw[fill=white] (-1.17,2) circle (10pt);
			\draw[fill=white] (1,0) circle (10pt);
			\draw[fill=white] (1,4) circle (10pt);
			\draw[fill=white] (2.67,1) circle (10pt);
			\draw[fill=white] (-.67,3) circle (10pt);
			\draw[fill=white] (2.67,3) circle (10pt);
			\draw[fill=white] (1,2) circle (10pt);
			\draw[fill=white] (-.67,1) circle (10pt);
			
			\node at (1,2) {$7$};
			\node at (0,0) {$1$};
			\node at (2,0) {$9$};
			\node at (3.17,2) {$2$};
			\node at (2,4) {$13$};
			\node at (0,4) {$5$};
			\node at (-1.17,2) {$12$};
			\node at (1,0) {$11$};
			\node at (1,4) {$3$};
			\node at (2.67,1) {$10$};
			\node at (-.67,3) {$4$};
			\node at (2.67,3) {$6$};
			\node at (-.67,1) {$8$};
			\end{tikzpicture}
			\caption{A magic hexagon}
			\label{fig:magichexagon}
		\end{figure}
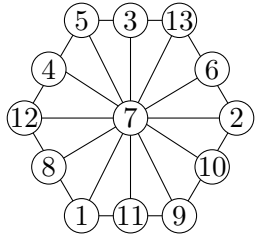
	\end{example}
	
	For those with experience in graph theory, the arrangement of nodes in a magic polygon corresponds to the vertices in a wheel graph.
	
	%-------------------------
	\section{Properties of Magic Polygons}
	%-------------------------
	
	In this section, we investigate the properties that magic polygons must have.  To begin, we consider what number from the set $\{1,2,\ldots, 2n+1\}$ must be placed in the center  node.

	\begin{proposition}\label{centernumber}
		Let $n\in\mathbb{N}\setminus\{1,2\}$.  If a magic $n-gon$ exists, then the center number is $n+1$.
	\end{proposition}
	\begin{proof}
		Consider a regular $n$-gon. Then, the numbers to be placed in the nodes of the $n$-gon are
		\begin{equation*}
		S = \{1,2,3,...,n,n+1,n+2,...,2n-1,2n,2n+1\}.
		\end{equation*}
		Let $c$ be the number in the center of the magic polygon. Then, considering the $n$ diagonals, there are $n$ distinct pairs of numbers summing to $m - c$, where $m$ denotes the magic sum. Since $c \in S$, then $c \in \{ 1, n+1, 2n+1\}$.  Furthermore, the orientation of a diagonal of an $n$-gon is dependent on the parity of $n$ and can be illustrated as:
		
		\begin{figure}[h]
			\begin{subfigure}{.25\textwidth}
				\centering
				\begin{tikzpicture}[scale=.65]
				\draw (0,0)--(2,0);
				\draw (1,0)--(1,2.5);
				\draw (1,2.5)--(3,2);
				\draw (1,2.5)--(-1,2);
				\draw (0,2.25)--(2,0);
				\draw (0,0)--(2,2.25);
				\draw[fill=white] (0,0) circle (10pt);
				\draw[fill=white] (1,0) circle (10pt);
				\draw[fill=white] (2,0) circle (10pt);
				\draw[fill=white] (1,2.5) circle (10pt);
				\draw[fill=white] (3,2) circle (10pt);
				\draw[fill=white] (-1,2) circle (10pt);
				\draw[fill=white] (2,2.25) circle (10pt);
				\draw[fill=white] (0,2.25) circle (10pt);
				\draw[fill=white] (1,1.25) circle (10pt);
				
				\node at (0,0) {$a$};
				\node at (1,0) {$y$};
				\node at (2,0) {$b$};
				\node at (1,1.25) {$c$};
				\node at (1,2.5) {$x$};
				%    \node at (3,2) {$g$};
				\node at (-1,2) {$d$};
				%    \node at (2,2.25) {$f$};
				\node at (0,2.25) {$e$};
				\end{tikzpicture}
				\caption{$n$ odd}
				\label{fig:odd}
			\end{subfigure}%
			\begin{subfigure}{.25\textwidth}
				\centering
				\begin{tikzpicture}[scale=.65]
				\draw (0,0)--(3,0)--(0,3)--(3,3)--(0,0);
				\draw (1.5,0)--(1.5,3);
				\draw[fill=white] (0,0) circle (10pt);
				\draw[fill=white] (1.5,0) circle (10pt);
				\draw[fill=white] (3,0) circle (10pt);
				\draw[fill=white] (1.5,1.5) circle (10pt);
				\draw[fill=white] (0,3) circle (10pt);
				\draw[fill=white] (1.5,3) circle (10pt);
				\draw[fill=white] (3,3) circle (10pt);
				\node at (0,3) {$x$};
				\node at (1.5,3) {$y$};
				\node at (1.5,1.5) {$c$};
				\node at (0,0) {$d$};
				\node at (1.5,0) {$e$};
				\node at (3,0) {$f$};
				\node at (3,3) {$g$};
				\end{tikzpicture}
				\subcaption{$n$ even}
				\label{fig:even}
			\end{subfigure}
			\caption{Layout of opposing edges when $n$ is odd versus $n$ is even}
			\label{fig:test}
		\end{figure}
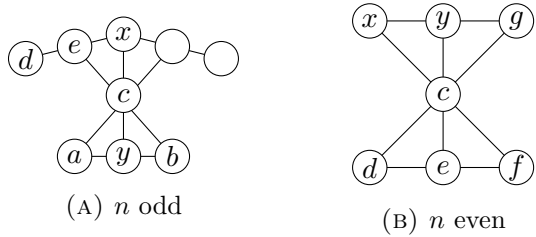
		
		\noindent Now, we consider two cases.
		
		\begin{propcase}
			Assume $c=1$.  Then, the $n$ distinct pairs of numbers make up the elements in the set $\{(2,2n+1),(3,2n),\ldots ,(n+1,n+2)\}$, and the magic sum is $m=2n+4$.  Consider the placement of the pair $(2,2n+1)$ on one of the diagonals.  If $n$ is odd, then using the labeling in Figure~\ref{fig:odd}, either $x=2$ or $y=2$.  If $x=2$,  then $a+b=(2n+4)-(2n+1)=3$.  However, $a,b\in S\setminus\{1,2,2n+1\}$, a contradiction.  If $y=2$, then similarly $d+e=3$, a contradiction. Thus, the pair $(2,2n+1)$ cannot be placed on a diagonal when $n$ is odd.  If $n$ is even, then using the labeling in Figure~\ref{fig:even}, either $x=2$ or $y=2$.  If $x=2$, then $d+e=(2n+4)-(2n+1)=3$, a contradiction.  If $y=2$, then similarly $d+f=3$, a contradiction.  Thus, the pair $(2,2n+1)$ cannot be placed on a diagonal when $n$ is even.  Therefore, no magic $n$-gon exists if $c=1$.
		\end{propcase}
		
		\begin{propcase}
			Assume $c=2n+1$.  Then, the $n$ distinct pairs of numbers are the elements in the set $\{(1,2n),(2,2n-1),\ldots ,(n,n+1)\}$, and the magic sum is $m=4n+2$.  Consider the placement of the pair $(1,2n)$ on one of the diagonals.  If $n$ is odd, then using the labeling in Figure~\ref{fig:odd}, either $x=1$ or $y=1$.  If $x=1$, then $d+e=(4n+2)-1=4n+1$.  However, $d,e\in S\setminus\{1,2n,2n+1\}$, so $d+e\leq (2n-1)+(2n-2)=4n-3$, a contradiction.  If $y=1$, then similarly $a+b=4n+1$, a contradiction.   Thus, the pair $(1,2n)$ cannot be placed on a diagonal when $n$ is odd.  If $n$ is even, then using the labeling in Figure~\ref{fig:even}, either $x=1$ or $y=1$.  If $x=1$, then $g+y=(4n+2)-1=4n+1$, a contradiction.  If $y=1$, then similarly $x+g=4n+1$, a contradiction.  Hence, the pair $(1,2n)$ cannot be placed on a diagonal when $n$ is even.  Therefore, no magic $n$-gon exists if $c=2n+1$.
		\end{propcase}
		
		\noindent Therefore, if a magic $n$-gon exists, $c=n+1$.
	\end{proof}
	
	Next, we investigate what the magic sum, $m$, must be in a magic polygon.
	
	\begin{proposition}\label{magicsum}
		Let $n\in\mathbb{N}\setminus\{1,2\}$.  If a magic $n$-gon exists, the magic sum is $m=3n+3$.
	\end{proposition}
	\begin{proof}
		Assume a magic $n$-gon exists.  Then, the $n$ diagonals all sum to the magic sum, $m$.  Since each of diagonals contains the center number, $c$, Proposition~\ref{centernumber} implies there are $n$ distinct pairs of numbers summing to 
		\[m-c=m-(n+1)=m-n-1.\]
		Then, the nodes placed on the edges of the polygon (i.e.\ all nodes except the center node) sum to 
		\begin{equation}\label{sum1}
		n(m-n-1)=nm-n^2-n.
		\end{equation}
		Alternatively, this sum could be computed as
		\begin{equation}\label{sum2}
		\displaystyle{\left(\sum_{k=1}^{2n+1} k\right)-(n+1) = \frac{(2n+1)(2n+2)}{2}-n-1=2n^2+2n.}
		\end{equation}
		Setting Equation~\ref{sum1} and Equation~\ref{sum2} equal to each other produces $m=3n+3$.
	\end{proof}
	
	%-------------------------
	\section{Existence of Magic Polygons}
	%-------------------------
	
	In this section, we present the main results of the paper pertaining to the existence of magic polygons. For a regular $n$-gon, since  the placement of the diagonals varies based on the parity of $n$, it makes sense to consider the magic polygons in two distinct classes based on whether $n$ is even or odd.
	
	\begin{theorem}
		Let $n\in\mathbb{N}$ be odd, with $n\geq 3$. Then no such magic $n$-gon exists.
	\end{theorem}
	
	\begin{proof}
		Let $n=2k+1$ where $k\in\mathbb{N}$, and assume a magic $n$-gon exists.   From Proposition~\ref{centernumber}, the center number is $c=n+1=2k+2$; and from Proposition~\ref{magicsum}, the magic sum is $m=3n+3=6k+6$.  Then the $n$ distinct pairs of numbers that must be placed onto the $n$ diagonals sum to $m-c=(6k+6)-(2k+2)=4k+4$ and are in the set $P=\{(1,4k+3),(2,4k+2),\ldots, (2k+1,2k+3)\}$.  Consider the placement of the pair $(1,4k+3)$ on a diagonal.   Since $n$ is odd, each diagonal contains a vertex of the $n$-gon as well as the midpoint of the opposing side, depicted as:
		
		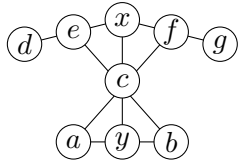
\begin{figure}[h]
			\centering
			\begin{tikzpicture}[scale=.65]
			\draw (0,0)--(2,0);
			\draw (1,0)--(1,2.5);
			\draw (1,2.5)--(3,2);
			\draw (1,2.5)--(-1,2);
			\draw(0,2.25)--(2,0);
			\draw (0,0)--(2,2.25);
			\draw[fill=white] (0,0) circle (10pt);
			\draw[fill=white] (1,0) circle (10pt);
			\draw[fill=white] (2,0) circle (10pt);
			\draw[fill=white] (1,2.5) circle (10pt);
			\draw[fill=white] (3,2) circle (10pt);
			\draw[fill=white] (-1,2) circle (10pt);
			\draw[fill=white] (2,2.25) circle (10pt);
			\draw[fill=white] (0,2.25) circle (10pt);
			\draw[fill=white] (1,1.25) circle (10pt);
			
			\node at (0,0) {$a$};
			\node at (1,0) {$y$};
			\node at (2,0) {$b$};
			\node at (1,1.25) {$c$};
			\node at (1,2.5) {$x$};
			\node at (3,2) {$g$};
			\node at (-1,2) {$d$};
			\node at (2,2.25) {$f$};
			\node at (0,2.25) {$e$};
			\end{tikzpicture}
			\caption{Layout of a diagonal when $n$ is odd.}
			\label{fig:odddiagonal}
		\end{figure}
		
		\noindent Then, using the labeling of Figure~\ref{fig:odddiagonal}, either $x=1$ or $x=4k+3$.
		\begin{theoremcase}
			Assume $x=1$.  Then $y=4k+3$, and we get the following system of 5 linear equations:
			\[\hspace{1in}d+e=f+g=6k+5\hspace{.25in} a+b=2k+3\hspace{.25in} b+e=a+f=4k+4.\]
			Solving using elementary row operations on an augmented matrix whose columns correspond to $a$, $b$, $d$, $e$, and $f$ (respectively) produces the following:
			\[\hspace{1in}\begin{array}{rcl}\left[\begin{array}{rrrrrrr}
			0 & 0 & 1 & 1 & 0 & 0 & 6k+5\\
			0 & 0 & 0 & 0 & 1 & 1 & 6k+5\\
			1 & 1 & 0 & 0 & 0 & 0 & 2k+3\\
			0 & 1 & 0 & 1 & 0 & 0 & 4k+4\\
			1 & 0 & 0 & 0 & 1 & 0 & 4k+4\\
			\end{array}\right]  & \xrightarrow{\mbox{RREF}} & 
			\left[\begin{array}{rrrrrrr}
			1 & 0 & 0 & 0 & 0 & -1 & -2k-1\\
			0 & 1 & 0 & 0 & 0 & 1 & 4k+4\\
			0 & 0 & 1 & 0 & 0 & 1 & 6k+5\\
			0 & 0 & 0 & 1 & 0 & -1 & 0\\
			0 & 0 & 0 & 0 & 1 & 1 & 6k+5\\
			\end{array}\right].\end{array}\]
			The fourth row implies that $e=g$.  Hence, $x\neq 1$.
		\end{theoremcase}
		
		\begin{theoremcase}
			Assume $x=4k+3$.  Then $y=1$, and we get the following system of 5 linear equations:
			\[\hspace{1in}d+e=f+g=2k+3\hspace{.25in}a+b=6k+5\hspace{.25in}b+e=a+f=4k+4.\]
			As in the case when $x=1$, we find that this system is consistent, and all solutions must have $e=g$.  Hence, $x\neq 4k+3$.
		\end{theoremcase}
		\noindent Thus, the pair $(1,4k+3)$ cannot be placed.  Hence, a magic $n$-gon does not exist when $n$ is odd.
	\end{proof}
	
	It remains to consider $n$-gons where $n$ is even.  For the remainder of the paper, we will denote the integers placed in the nodes at the vertices of the $n$-gon, indexed in a clockwise order, by $v_1,\ldots, v_n$. 
	
	\begin{lemma}\label{function}
		Let $n\in\mathbb{N}$ be even with $n\geq 6$.  Then, the map 
		\[f:\left\{v_i : i\in\mathbb{N}\mbox{ and }1\leq i\leq \frac{n}{2}\right\} \rightarrow \{1,2,\ldots, 2n+1\}\]
		where $f(v_1)=n-1$, $f(v_2)=2n+1$, and for $i \geq 3$: 
		\[f(v_i)=\left\{\begin{array}{rl} n+i & ,i=2\ell\\ i-1 & ,i=2\ell +1\end{array}\right.\]
		is a well-defined function. 
	\end{lemma}
	\begin{proof}  
		Let $v_i\in\left\{v_i : i\in\mathbb{N}\mbox{ and }1\leq i\leq \frac{n}{2}\right\}$.  If $i=1,2$, then $f(v_i)\in \{1,2,\ldots, 2n+1\}$.  Assume $i\geq 3$.
		\begin{lemmacase}
			If $i=2\ell$, then $f(v_i)=n+i\in\mathbb{N}$.  Moreover, $f(v_i)=n+i\leq n+\frac{n}{2}\leq 2n$ and $f(v_i)\in\{1,2,\ldots, 2n+1\}$.
		\end{lemmacase}
		\begin{lemmacase}
			If $i=2\ell +1$, then $f(v_i)=i-1\in\mathbb{N}$ since $i\geq 3$.  Moreover, $f(v_i)=i-1= 2\ell +1 -1 =2\ell\leq \frac{n}{2}-1$ and $f(v_i)\in\{1,2,\ldots, 2n+1\}$.
		\end{lemmacase}
		\noindent Therefore, $f$ is a well-defined function.
	\end{proof}
	
	It should be noted that to complete a magic $n$-gon, determining the values to be placed in $v_1, v_2,\ldots , v_\frac{n}{2},$ along with the center number, is sufficient because, for $1\leq i\leq \frac{n}{2}$, applying Propositions~\ref{centernumber} and~\ref{magicsum} results in an extension of the function $f$ in Lemma~\ref{function} to a map $f^*$ where:
	\begin{equation}\label{ext1}
	\begin{array}{rcl}
	f^*(v_{\frac{n}{2}+i}) & = & m-c-f(v_i) = (3n+3) - (n+1)-f(v_i) = 2n+2 - f(v_i)\\
	\end{array}
	\end{equation}
	Additionally, for $1\leq i\leq n$, we will use $m_i$ to denote the node located at the midpoint of the edge labeled by $v_i$ and $v_{i+1}$, where $v_{n+1}:=v_1$.  With this notation, for $1\leq i\leq n$:
	\begin{equation}\label{ext2}
	f^*(m_i) = m - f^*(v_i)-f^*(v_{i+1}) = 3n+3-f^*(v_i)-f^*(v_{i+1}).
	\end{equation}
	
	\begin{corollary}\label{functionext} 
		Let $n\in\mathbb{N}$ be even with $n\geq 6$.  Then, the map 
		\[f:\left\{v_i : i\in\mathbb{N}\mbox{ and }1\leq i\leq \frac{n}{2}\right\} \rightarrow \{1,2,\ldots, 2n+1\}\]
		from Lemma~\ref{function} can be extended to a map 
		\[f^*:\{v_i :i\in\mathbb{N}\mbox{ and }1\leq i\leq n\}\cup\{m_i:i\in\mathbb{N}\mbox{ and }1\leq i\leq n\}\cup\{c\} \rightarrow \{1,2,\ldots, 2n+1\}\]
		using Equations~\ref{ext1} and~\ref{ext2}.  Moreover, $f^*$ is a well-defined function.
	\end{corollary}
	\begin{proof}
		The result follows immediately from Lemma~\ref{function}, Lemma~\ref{centernumber}, Equation~\ref{ext1}, and Equation~\ref{ext2}.
	\end{proof}
	
	As we proceed, it should be noted that, with a slight abuse of notation, $c$ will be used to denote both the the node at the center of the polygon as well as the value $f^*(c)=n+1$ as in Proposition~\ref{centernumber}.  
	
	The remainder of the paper will show that the function $f^*$ produces a magic polygon when $n$ is even.  Since 
	\[|\{v_i :i\in\mathbb{N}\mbox{ and }1\leq i\leq n\}\cup\{m_i:i\in\mathbb{N}\mbox{ and }1\leq i\leq n\}\cup\{c\}|=|\{1,2,\ldots ,2n+1\}|=2n+1,\]  
	it will suffice to show that $f^*$ is one-to-one.
	
	\begin{lemma}\label{Llemma}
		Let $n\in\mathbb{N}$ be even with $n\geq 6$.  For $\ell_1,\ell_2, j \in\mathbb{N}$ with $4\leq 2\ell_1,2\ell_2+1\leq\frac{n}{2}$ and $3\leq j<\frac{n}{2}$, using the function $f^*$ from Corollary~\ref{functionext}, 
		\[f^*(v_{\frac{n}{2}+2\ell_1}), f^*(v_{2\ell_2+1}), f^*(m_{\frac{n}{2}+j})\in\{4,5\ldots, n-2\}.\]
	\end{lemma}
	\begin{proof}  From Lemma~\ref{function} and Equations~\ref{ext1} and~\ref{ext2}, we have that
		\[\begin{array}{rcl}
		v_{\frac{n}{2}+2\ell_1} & = & (2n+2)-(n+2\ell_1)=n+2-2\ell_1\\
		v_{2\ell_2+1} & = & 2\ell_2+1-1=2\ell_2 \\
		\end{array}.\]
		Since $4\leq 2\ell_1,2\ell_2+1\leq\frac{n}{2}$, it follows that $\frac{n}{2}+2\leq f^*(v_{\frac{n}{2}+2\ell_1})\leq n-2$ and $3\leq f^*(v_{2\ell_2+1})\leq\frac{n}{2}+2$.  Since $n\geq 6$, $5 \leq f^*(v_{\frac{n}{2}+2\ell_1})\leq n-2$; and since $3\neq 2\ell_2$, $4\leq f^*(v_{2\ell_2+1})\leq\frac{n}{2}+2$.  Hence, $f^*(v_{\frac{n}{2}+2\ell_1}), f^*(v_{2\ell_2+1})\in\mathcal{L}$.  Moreover, from Equations~\ref{ext1} and \ref{ext2}:
		\[\begin{array}{rcl}
		f^*(m_{\frac{n}{2}+j}) & = & m-(2n+2-f^*(v_i))-(2n+2-f^*(v_{i+1}))\\
		& = & \left\{\begin{array}{cl}-n-1+(n+i)+(i+1-1) &,i\mbox{ is even}\\
		-n-1+(i-1)+(n+i+1) & ,i\mbox{ is odd}\end{array}\right.\\
		& = & 2i-1.\end{array}\]
		
		\noindent Since, $4\leq i\leq\frac{n}{2}$, $7\leq f^*(m_{\frac{n}{2}+j})\leq\frac{n}{2}$.  Hence, $f^*(m_{\frac{n}{2}+j})\leq\frac{n}{2}\in\mathcal{L}$.
	\end{proof}
	
	\begin{corollary}\label{Ucor}
		Let $n\in\mathbb{N}$ be even with $n\geq 6$.  For $\ell_1,\ell_2, j \in\mathbb{N}$ with $4\leq 2\ell_1,2\ell_2+1\leq\frac{n}{2}$ and $3\leq j<\frac{n}{2}$, using the function $f^*$ from Corollary~\ref{functionext}, 
		\[f^*(v_{2\ell_1}), f^*(v_{\frac{n}{2}+2\ell_2+1}), f^*(m_j)\in\{n+4,n+5,\ldots, 2n-2\}.\]
	\end{corollary}
	\begin{proof}
		This follows immediately from Lemma~\ref{Llemma} and Equations~\ref{ext1} and~\ref{ext2}.
	\end{proof}

	\begin{theorem}
		Consider $n=2k$ where $k\in\mathbb{N}\setminus\{1\}$.  Then a magic $n$-gon exists.
	\end{theorem}
	\begin{proof}
		From Figures~\ref{fig:magicpolygonsquare} and~\ref{fig:magichexagon}, it is clear that magic $n$-gons exist when $n=4,6$.  Moreover, the magic hexagon in Example~\ref{hexagon} can be obtained using the assignment of values to the nodes given by the function $f^*$ in Corollary~\ref{functionext}.  Assume $n\geq 8$.  We will show that the function $f^*$ defined in Corollary~\ref{functionext} is a one-to-one map.  Consider the initial assignment of $f^*(v_1)$, $f^*(v_2)$, and $f^*(v_3)$ using the function $f$ in Lemma~\ref{function}.  Then, we are able to assign values to the nodes in the set $\{v_1,v_2,v_3,c,v_{\frac{n}{2}+1},v_{\frac{n}{2}+2},v_{\frac{n}{2}+3},m_1,m_2,m_{\frac{n}{2}+1},m_{\frac{n}{2}+2}\}$ using Equations~\ref{ext1} and ~\ref{ext2}. Morever,  
		\[\{f^*(v_1),f^*(v_2),f^*(v_3),c,f^*(v_{\frac{n}{2}+1}),f^*(v_{\frac{n}{2}+2}),f^*(v_{\frac{n}{2}+3}),f^*(m_1),f^*(m_2),f^*(m_{\frac{n}{2}+1}),f^*(m_{\frac{n}{2}+2})\}\hspace{.25in} \]
		\[\hspace{3in}=\{n-1,2n+1,2,n+1,n+3,1,2n,3,n,2n-1,n+2\}.\]
		Hence, all other nodes must be assigned values from the set:
		\[\begin{array}{rcl}
		P & = & \{1,2,\ldots, 2n+1\}-\{n-1,2n+1,2,n+1,n+3,1,2n,3,n,2n-1,n+2\} \\
		& = & \{4,5,\ldots , n-2\}\cup\{n+4,n+5,\ldots, 2n-2\}\\
		\end{array}.\]
		Set $\mathcal{L}=\{4,5,\ldots , n-2\}$ and $\mathcal{U}=\{n+4,n+5,\ldots, 2n-2\}$.  
		
		From Lemma~\ref{Llemma} and Corollary~\ref{Ucor} it follows that for $\ell_1,\ell_2, j\in\mathbb{N}$ with $4\leq 2\ell_1,2\ell_2+1\leq\frac{n}{2}$ and $3\leq j<\frac{n}{2}$:
		\[\begin{array}{rcl}
		f^*(v_{\frac{n}{2}+2\ell_1}), f^*(v_{2\ell_2+1}),f^*(m_{\frac{n}{2}+j}) & \in & \mathcal{L}\\
		f^*(v_{2\ell_1}),f^*(v_{\frac{n}{2}+2\ell_2+1}), f^*(m_j) & \in & \mathcal{U}\\
		\end{array}.\]
		It remains to show that the elements in 
		\[P_{\mathcal{L}}:=\{f^*(v_{\frac{n}{2}+2\ell_1}), f^*(v_{2\ell_2+1}),f^*(m_{\frac{n}{2}+j}): 4\leq 2\ell_1,2\ell_2+1\leq\frac{n}{2}\mbox{ and }3\leq j<\frac{n}{2}\}\]
		are distinct for each allowable choice of $\ell_1,\ell_2,j\in\mathbb{N}$, and the elements in 
		\[P_{\mathcal{U}}:=\{f^*(v_{2\ell_1}),f^*(v_{\frac{n}{2}+2\ell_2+1}), f^*(m_j): 4\leq 2\ell_1,2\ell_2+1\leq\frac{n}{2}\mbox{ and }3\leq j<\frac{n}{2} \}\]
		are distinct  for each allowable choice of $\ell_1,\ell_2,j\in\mathbb{N}$.  Consider the set $P_{\mathcal{L}}$.
		\begin{theoremcase}
			Assume $f^*(v_{\frac{n}{2}+2\ell_1})=f^*(v_{2\ell_2+1})$.  Then, $n+2-2\ell_1=(2\ell_2+1)-1$, which implies that $n=2\ell_1+(2\ell_2+1)-2\leq\frac{n}{2}+\frac{n}{2}-2=n-2$, a contradiction.  Hence, $f^*(v_{\frac{n}{2}+2\ell_1})\neq f^*(v_{2\ell_2+1})$.
		\end{theoremcase}
		\begin{theoremcase}
			Assume $f^*(v_{\frac{n}{2}+2\ell_1})=f^*(m_{\frac{n}{2}+j})$.  Then, $n+2-2\ell_1=n-1$, and it follows that $\ell_1=\frac{3}{2}$, a contradiction.  Therefore, $f^*(v_{\frac{n}{2}+2\ell_1})\neq f^*(m_{\frac{n}{2}+j})$.
		\end{theoremcase}
		\begin{theoremcase}
			Assume $f^*(v_{2\ell_2+1})=f^*(m_{\frac{n}{2}+j})$.  Then, $2\ell_2+1-1=n-1$, which implies that $n=2\ell_2+1$, a contradiction.  Thus, $f^*(v_{2\ell_2+1})\neq f^*(m_{\frac{n}{2}+j})$.
		\end{theoremcase}
		\noindent Therefore, the elements of $P_{\mathcal{L}}$ are distinct for each allowable choice of $\ell_1,\ell_2,j\in\mathbb{N}$.  
		
		Next consider the set $P_{\mathcal{U}}$.  From the construction, if $f^*(v_i)\in \mathcal{L}$, then $f^*(v_{\frac{n}{2}+i}) \in \mathcal{U}$; and if $f^*(v_i)\in \mathcal{U}$, then $f^*(v_{\frac{n}{2}+i}) \in \mathcal{L}$.  Similarly, if $f^*(m_i)\in\mathcal{L}$, then $f^*(m_{\frac{n}{2}+i})\in\mathcal{U}$; and if $f^*(m_i)\in\mathcal{U}$, then $f^*(m_{\frac{n}{2}+i})\in\mathcal{L}$.  Then, it follows from Equations~\ref{ext1} and~\ref{ext2} that, since the elements of of $P_{\mathcal{L}}$ are distinct for each allowable choice of $\ell_1,\ell_2,j\in\mathbb{N}$, the elements of $P_{\mathcal{U}}$ are also distinct  for each allowable choice of $\ell_1,\ell_2,j\in\mathbb{N}$.
		
		Furthermore, $|P_{\mathcal{L}}|=n-6$ and $|P_{\mathcal{U}}|=n-6$.  Consider $m_{\frac{n}{2}}$ and $m_{n}$.  From Equation~\ref{ext2} and Lemma~\ref{function}:
		\[f^*(m_{\frac{n}{2}})=\left\{\begin{array}{cl} \frac{n}{2} & ,\mbox{ if }\frac{n}{2}\mbox{ is even}\\ \frac{3n}{2}+1 & ,\mbox{ if }\frac{n}{2}\mbox{ is odd}\end{array}\right.\hspace{.25in}\mbox{and}\hspace{.25in}f^*(m_n)=\left\{\begin{array}{lcl} \frac{3n}{2}+2 & ,\mbox{ if }\frac{n}{2}\mbox{ is even}\\ \frac{n}{2}+1 & ,\mbox{ if }\frac{n}{2}\mbox{ is odd}\end{array}\right.\]
		It follows that, if $\frac{n}{2}$ is even, then $f^*(m_{\frac{n}{2}})\in \mathcal{L}$ and $f^*(m_n)\in\mathcal{U}$; and if $\frac{n}{2}$ is odd, then $f^*(m_{\frac{n}{2}})\in \mathcal{U}$ and $f^*(m_n)\in\mathcal{L}$.  When $\frac{n}{2}$ is even, it is easy to verify that $f^*(m_{\frac{n}{2}})\notin P_{\mathcal{L}}$ and $f^*(m_n)\notin P_{\mathcal{U}}$; and when $\frac{n}{2}$ is odd, it is easy to verify that $f^*(m_{\frac{n}{2}})\notin P_{\mathcal{U}}$ and $f^*(m_n)\notin P_{\mathcal{L}}$.  Moreover, when $\frac{n}{2}$ is even $|P_{\mathcal{L}}\cup\{f^*(m_{\frac{n}{2}})\}|=n-5=|\mathcal{L}|$ and $|P_{\mathcal{U}}\cup\{f^*(m_n)\}|=n-5=|\mathcal{U}|$; and when when $\frac{n}{2}$ is odd $|P_{\mathcal{L}}\cup\{f^*(m_n)\}|=n-5=|\mathcal{L}|$ and $|P_{\mathcal{U}}\cup\{f^*(m_\frac{n}{2})\}|=n-5=|\mathcal{U}|$.
		
		It remains to show that for $1\leq j\leq\frac{n}{2}$, $f^*(m_j)+f^*(m_{\frac{n}{2}+j})=m-(n+1)=2n+2$.  If $j\neq\frac{n}{2}$, then $f^*(m_j)+f^*(m_{\frac{n}{2}+j})=(2n-2j+3)+(2j-1)=2n+2$.  If $j=\frac{n}{2}$, then 
		\[f^*(m_j)+f^*(m_{\frac{n}{2}+j})=\left\{\begin{array}{cl}
		\frac{n}{2}+\left(\frac{3n}{2}+2\right) & ,\mbox{ if }\frac{n}{2}\mbox{ is even}\\
		\left(\frac{3n}{2}+1\right)+\left(\frac{n}{2}+1\right)& ,\mbox{ if }\frac{n}{2}\mbox{ is odd}
		\end{array}\right. = 2n+2.\]
		Therefore, since the function $f^*$ in Corollary~\ref{functionext} is one-to-one and 
		\[|\{f^*(v_i) :i\in\mathbb{N}\mbox{ and }1\leq i\leq n\}\cup\{f^*(m_i):i\in\mathbb{N}\mbox{ and }1\leq i\leq n\}\cup\{c\}|\]
		\[\begin{array}{rcl}
		\hspace{1.25in}& = & |\mathcal{L}\cup\mathcal{U}\cup\{1,2,3,n-1,n,n+1,n+2,n+3,2n-1,2n,2n+1\}|\\
		& = & (n-5) + (n-5) + 11\\
		& = & 2n+1\\
		& = & |\{1,2,\ldots,2n+1\}|,\\
		\end{array}\]
		the function $f^*$ produces a magic $n$-gon when $n$ is even and $n\geq 6$. 
	\end{proof}
	
	\section{Conclusion}
	
	This extension of $3\times 3$ magic squares to magic polygons opens up a realm to be further explored. The function $f^*$ provided in Corollary~\ref{functionext} creates one possible magic $n$-gon when $n$ is even.  For magic squares, work has been done to enumerate the possibilities.  For instance when $n = 3$, there is a unique (up to symmetry) magic square (see~\cite{Swetz}), and the  number of solutions for $4 \times 4$ and $5 \times 5$ magic squares are $880$ and $275,305,224$, respectively (see~\cite{Loly}).  In the case of magic polygons, the enumeration of the distinct magic $n$-gons (up to symmetry) remains open.

\end{document}